\documentclass[12pt,reqno]{amsart}

\usepackage[left=3.5cm,right=3.5cm,top=3.5cm,bottom=3.5cm]{geometry}
\usepackage{amssymb}
\usepackage{graphicx}
\usepackage{amscd}
\usepackage[pagebackref]{hyperref}
\usepackage{color}
\usepackage{float}
\usepackage{graphics,amsmath,amssymb}
\usepackage{amsthm}
\usepackage{amsfonts}
\usepackage{latexsym}
\usepackage{epsf}
\usepackage{xifthen}
\usepackage{mathrsfs}
\usepackage{dsfont}
\usepackage{makecell}
\usepackage[FIGTOPCAP]{subfigure}
\usepackage{amsmath}
\allowdisplaybreaks[4]
\usepackage{listings}
\usepackage{etoolbox}
\usepackage{fancyhdr}
\usepackage{pdflscape}
\usepackage[title,toc,titletoc]{appendix}
\usepackage{enumitem}
\usepackage[noadjust]{cite}
\usepackage{cancel}
\usepackage{tabularx}
\usepackage[table]{xcolor}
\usepackage[title,toc,titletoc]{appendix}

\newtheoremstyle{mytheorem}% name % cf. thmtest.tex of AMSLaTeX
{3pt}% Space above
{3pt}% Space below
{\slshape}% Body font
{}% Indent amount (empty = no indent,
{\bfseries}% Thm head font
{.}% Punctuation after thm head
{ }% Space after thm head: " " = normal interword space;
{}% Thm head spec (can be left empty, meaning `normal')

\numberwithin{equation}{section}

\theoremstyle{theorem}
\newtheorem{theorem}{Theorem}[section]
\newtheorem*{theorem*}{Theorem}

\newtheorem{innercustomgeneric}{\customgenericname}
\providecommand{\customgenericname}{}
\newcommand{\newcustomtheorem}[2]{%
	\newenvironment{#1}[1]
	{%
		\renewcommand\customgenericname{#2}%
		\renewcommand\theinnercustomgeneric{##1}%
		\innercustomgeneric
	}
	{\endinnercustomgeneric}
}
\newcustomtheorem{ctheorem}{Theorem}
\newcustomtheorem{clemma}{Lemma}

\theoremstyle{definition}

\newtheorem*{example*}{Example}
\newtheorem*{examples*}{Examples}
\newtheorem{remark}{Remark}[section]
\newtheorem*{remark*}{Remark}
\newtheorem*{remarks*}{Remarks}

\newtheoremstyle{named}{}{}{\itshape}{}{\bfseries}{.}{.5em}{#1\thmnote{ #3}}
\theoremstyle{named}

\begin{document}

\title[Ramanujan-type series]{Proofs of conjectures on Ramanujan-type series of level 3}

\author[J. M. Campbell]{John M. Campbell}
\address[J. M. Campbell]{Department of Mathematics and Statistics, Dalhousie University, Halifax, NS, B3H 4R2, Canada}
\email{jmaxwellcampbell@gmail.com}

\date{}

\begin{abstract}	
 A Ramanujan-type series satisfies $$ \frac{1}{\pi} = \sum_{n=0}^{\infty} \frac{\left( \frac{1}{2} \right)_{n} \left( \frac{1}{s} \right)_{n} \left(1 - 
 \frac{1}{s} \right)_{n} }{ \left( 1 \right)_{n}^{3} } z^{n} (a + b n), $$ where $s \in \{ 2, 3, 4, 6 \}$, and where $a$, $b$, and $z$ are real algebraic 
 numbers. The level $3$ case whereby $s = 3$ has been considered as the most mysterious and the most challenging, out of all possible values for $s$, 
 and this motivates the development of new techniques for constructing Ramanujan-type series of level $3$. Chan and Liaw introduced an alternating 
 analogue of the Borwein brothers' identity for Ramanujan-type series of level $3$; subsequently, Chan, Liaw, and Tian formulated another proof of the 
 Chan--Liaw identity, via the use of Ramanujan's class invariant. Using the elliptic lambda function and the elliptic alpha function, we prove, using a 
 limiting case of the Kummer--Goursat transformation, a new identity for evaluating $z$, $a$, and $b$ for Ramanujan-type series such that $s = 3$ and 
 $z < 0$, and we apply this new identity to prove three conjectured formulas for quadratic-irrational, Ramanujan-type series that had been discovered via 
 numerical experiments with Maple in 2012 by Aldawoud. We also apply our identity to prove a new Ramanujan-type series of level $3$ with quartic values 
 for $z < 0$, $a$, and $b$. 
\end{abstract}



\maketitle
	
\section{Introduction}\label{sectionIntroduction}
 In his first letter to Hardy in January 1913, Ramanujan \cite[p.~25, Sect.~(5), eq.~(3)]{BR1995} presented the following remarkable series for $1/\pi$: 
\begin{equation}\label{slowestRamanujan}
	\frac{1}{\pi} = \sum_{n = 0}^{\infty} \frac{ \left( \frac{1}{2} \right)_{n}^3 }{ \left( 1 \right)_{n}^3} (-1)^{n} \big(\tfrac{1}{2}+2n\big).
\end{equation}
Soon after, what are referred to as Ramanujan's seventeen series for $1/\pi$ were introduced in \cite{Ramanujan1914} (see also \cite[pp.~36--38]{Ramanujan2000}), and have gone on, over the decades, to have a great influence on many different areas in mathematics. This motivates the development of mathematical topics based on the application of generalizations and variants of Ramanujan's seventeen series. Formally, we say a \emph{Ramanujan-type series} is of the form
\begin{equation}\label{maindefinition}
 \frac{1}{\pi} = \sum_{n = 0}^{\infty} \frac{\left( \frac{1}{2} \right)_{n} \left( \frac{1}{s} \right)_{n} 
 \left(1 - \frac{1}{s} \right)_{n} }{ \left( 1 \right)_{n}^{3} } z^n (a + b n), 
\end{equation}
 where $s \in \{ 2, 3, 4, 6 \}$ and where the parameters $a$, $b$, and $z$ are real algebraic numbers. This definition broadly agrees with relevant 
 references that have inspired our work, as in \cite{CooperZudilin2019,Guillera2020,Guillera2013,Guillera2006Exp,Guillera2021,Guillera2019,Wan2014}. 
 Historically, Ramanujan-type series may be classified by their \emph{level} $\ell$, under the mapping $\ell:=4\sin^2\frac{\pi}{s}$, so as to reach a 
 correspondence with the levels of the modular forms that parameterize Ramanujan-type series. See, for instance, \cite{Guillera2021}. The work of Guillera 
 \cite{Guillerafastest2021} on the level 3 case of \eqref{maindefinition} has led us to investigate new techniques for proving closed-form evaluations, 
 in the $s = 3$ case of \eqref{maindefinition}, for the expressions $z$, $a$, and $b$ in \eqref{maindefinition}. The alternative bases and modular 
 relations associated with the level $3$ case of \eqref{maindefinition} have been considered as the most important \cite{BBG1995} and most interesting 
 \cite{ChanGeeTan2003} out of all possible levels for Ramanujan-type series, which motivates the results and techniques we introduce concerning 
 Ramanujan-type series of level $3$. 

 Subsequent to the seminal monograph \emph{Pi and the AGM} \cite{BorweinBorwein1987text} by the Borwein brothers and to further groundbreaking works by the Borweins on Ramanujan-type series \cite{BorweinBorwein1993,BorweinBorwein1988,BorweinBorwein1987rational,BorweinBorwein1992}, an especially notable reference that introduced a large number of Ramanujan-type series is due to Baruah and Berndt \cite{BaruahBerndt2010}, who derived their results 
 using techniques based on Eisenstein series. In \cite{BaruahBerndt2010}, Baruah and Berndt's Ramanujan-type formulas are restricted to the scenarios 
 where $s$ is among $\{ 2, 4, 6 \}$. This greatly motivates the problem as to how new Ramanujan-type series in the vein of the results from 
 \cite{BaruahBerndt2010}, in which the $z$- and $a$- and $b$-values in \eqref{maindefinition} are often quadratic or quartic, could be determined for 
 the $s = 3$ case. This forms a main purpose of our article, in which we: 

\begin{enumerate}

\item Prove a new identity, in terms of the elliptic lambda and elliptic alpha functions, for evaluating $z$, $a$, and $b$ in \eqref{maindefinition}, for the 
 case whereby $s = 3$ and $z < 0$, with our techniques being fundamentally different compared to the work of Chan and Liaw \cite{ChanLiaw2000} and 
 of Chan, Liaw, and Tan \cite{ChanLiawTan2001} on the $s = 3$ and $z < 0$ case of \eqref{maindefinition}; 

\item Apply our identity for evaluating $z < 0$, $a$, and $b$ to prove three conjectured formulas on quadratic-irrational, Ramanujan-type series that had 
 been discovered experimentally via numerical experiments with Maple by Aldawoud in 2012 \cite{Aldawoud2012}; and 

\item Apply our identity to prove an evaluation for a new Ramanujan-type series of level $3$, with quartic values for $z < 0$, $a$, and $b$. 

\end{enumerate}
 
\section{Background and preliminaries}
 The \emph{Pochhammer symbol} is defined by 
\begin{align*}
 (x)_n:=\begin{cases}
 1, & n=0,\\
 x(x+1) \cdots (x + n - 1), & n\in\mathbb{N}.
 \end{cases}
\end{align*}
 We let ${}_{2}F_{1}$-series be denoted and defined so that 
\begin{equation*}
 {}_{2}F_{1}\!\!\left[ \begin{matrix} 
 a, b \vspace{1mm}\\ 
 c \end{matrix} \ \Bigg| \ x 
 \right] = \sum_{n=0}^{\infty} \frac{ \left( a \right)_{n} \left( b \right)_{n} }{ \left( c \right)_{n} } \frac{x^n}{n!}. 
\end{equation*}
 A cubic series due to the Borwein brothers was explored by Chan and Liaw in 
 \cite{ChanLiaw2000}, and what may be regarded as an alternating analogue of this cubic series due to Chan and Liaw 
 \cite{ChanLiaw2000} was reproved in \cite{ChanLiawTan2001}. Following \cite{ChanLiawTan2001}, we set 
\begin{equation}\label{mquotient} 
 m = \frac{ {}_{2}F_{1}\!\!\left[ \begin{matrix} 
 \frac{1}{3}, \frac{2}{3} \vspace{1mm}\\ 
 1 \end{matrix} \ \Bigg| \ \alpha 
 \right] }{ {}_{2}F_{1}\!\!\left[ \begin{matrix} 
 \frac{1}{3}, \frac{2}{3} \vspace{1mm}\\ 
 1 \end{matrix} \ \Bigg| \ \beta 
 \right]}, 
\end{equation}
 and we are letting $\alpha$ and $\beta$ be such that $$ \frac{ {}_{2}F_{1}\!\!\left[ \begin{matrix} 
 \frac{1}{3}, \frac{2}{3} \vspace{1mm}\\ 
 1 \end{matrix} \ \Bigg| \ 1 - \beta 
 \right] }{ {}_{2}F_{1}\!\!\left[ \begin{matrix} 
 \frac{1}{3}, \frac{2}{3} \vspace{1mm}\\ 
 1 \end{matrix} \ \Bigg| \ \beta 
 \right] } 
 = n \frac{ {}_{2}F_{1}\!\!\left[ \begin{matrix} 
 \frac{1}{3}, \frac{2}{3} \vspace{1mm}\\ 
 1 \end{matrix} \ \Bigg| \ 1 - \alpha 
 \right] }{ {}_{2}F_{1}\!\!\left[ \begin{matrix} 
 \frac{1}{3}, \frac{2}{3} \vspace{1mm}\\ 
 1 \end{matrix} \ \Bigg| \ \alpha 
 \right] }.$$
 In something of a similar fashion, the \emph{cubic singular modulus} is the unique value $\alpha_{n}$ 
 such that 
 \begin{equation}\label{cubicalpha}
 \frac{ {}_{2}F_{1}\!\!\left[ \begin{matrix} 
 \frac{1}{3}, \frac{2}{3} \vspace{1mm}\\ 
 1 \end{matrix} \ \Bigg| \ 1 - \alpha_{n} 
 \right] }{ {}_{2}F_{1}\!\!\left[ \begin{matrix} 
 \frac{1}{3}, \frac{2}{3} \vspace{1mm}\\ 
 1 \end{matrix} \ \Bigg| \ \alpha_{n} 
 \right] } = \sqrt{n}. 
\end{equation}
 The Borwein brothers' family of Ramanujan-type series of level 3 is such that the convergence rate $\mathcal{H}_{n}$ satisfies $\mathcal{H}_{n} = 4 
 \alpha_{n}(1- \alpha_{n})$, referring to \cite{ChanLiawTan2001} for details. From \eqref{cubicalpha}, this convergence rate is necessarily positive for 
 $n > 0$. This motivates us to determine an explicit way of evaluating the $z$- and $a$- and $b$-expressions in \eqref{maindefinition} for the $s = 3$ 
 and $z < 0$ case, and with the use of the elliptic lambda and elliptic alpha functions given in the \emph{Pi and the AGM} 
 text \cite{BorweinBorwein1987text}. 

 The infinite family of Ramanujan-type series under consideration in \cite{ChanLiawTan2001} is such that 
\begin{equation}\label{CLTfamily}
 \frac{1}{2\pi} \sqrt{\frac{3}{n}} \sum_{k=0}^{\infty} (a_{n} + b_{n} k) 
 \frac{ \left( \frac{1}{2} \right)_{k} 
 \left( \frac{1}{3} \right)_{k} \left( \frac{2}{3} \right)_{k} }{ \left( k! \right)^{3} } \mathcal{H}_{n}^{k}, 
\end{equation}
 where 
\begin{equation}\label{bada}
 a_{n} = \frac{ \alpha_{n} (1-\alpha_{n}) }{\sqrt{n}} \frac{dm}{d\alpha}(1 - \alpha_{n}, \alpha_{n}) 
\end{equation}
 and $ b_{n} = 1 - 2 \alpha_{n} $ and $ 
 \mathcal{H}_{n} = 4 \alpha_{n} (1 - \alpha_{n})$; see 
 also \cite{Guillerafastest2021} for a family of Ramanujan-type series that is similarly formulated in terms of a 
 transformation of modular origin, as in the modular function in \eqref{mquotient}. This family of Ramanujan-type series given by Guillera 
 \cite{Guillerafastest2021} was derived using material due to Guillera in \cite{Guillera2020} and Wan in \cite{Wan2014}. 

 As expressed in \cite{ChanLiaw2000}, the evaluation of expressions as in \eqref{bada} is very nontrivial. The main difference between previously 
 known formulas for evaluating summands as in \eqref{CLTfamily} and our new technique for determining closed forms for $z$, $a$, and $b$ in 
 \eqref{maindefinition} for the case such that $z < 0$ and $s = 3$ is given by how we do not require the derivative of the $m$-function shown in 
 \eqref{bada}. This has an advantage given by how we may evaluate the equivalent of the $a_{n}$-expression in \eqref{bada} in a more explicit way 
 in terms of the elliptic alpha function introduced by the Borwein brothers \cite{BorweinBorwein1987text}, and this, in turn, has the advantage of allowing 
 us to apply previously known values for the elliptic alpha function. Furthermore, our method gives us a more explicit way of expressing the cubic 
 singular modulus in \eqref{cubicalpha}, and in a way that allows us to apply previously known closed forms for the elliptic lambda function 
 $\lambda^{\ast}$. This leads us toward our review, as below, of the complete elliptic integrals and of the elliptic lambda and elliptic alpha functions. 

 The \emph{complete elliptic integrals} of the first and second kinds \cite[Sect.~1.3]{BorweinBorwein1987text} are, respectively, defined by 
\begin{align*}
	\mathbf{K}(k) := \int_{0}^{\pi/2} \frac{d\theta}{\sqrt{1 - k^2 \sin^2 \theta}} \ \ \ \text{and} 
 \ \ \ \mathbf{E}(k) := \int_{0}^{\pi/2} \sqrt{1 - k^2 \sin^2 \theta} \, d\theta. 
\end{align*}
The argument $k$ indicated above is referred to as the \emph{modulus}, and it is often convenient to write $k' := \sqrt{1-k^2} $ to denote the \emph{complementary modulus}, so that we may write, accordingly, $ \mathbf{K}'(k) := \mathbf{K}(k')$ and $ \mathbf{E}'(k) := \mathbf{E}(k')$. By virtue of \emph{Jacobi's imaginary quadratic transformations}, the two complete elliptic integrals satisfy the following modular equations \cite[p.~72, Exs.~7--8]{BorweinBorwein1987text}:
\begin{align}\label{maincomplex}
 \mathbf{K}\left(i\frac{k}{k'}\right) = k'\mathbf{K}(k) \ \ \ 
 \text{and} \ \ \ \mathbf{E}\left(i\frac{k}{k'}\right) = \frac{1}{k'}\mathbf{E}(k). 
\end{align}
Furthermore, the differential equation such that 
\begin{equation}\label{mainODE}
	\frac{d\mathbf{K}}{dk} 
	= \frac{ \mathbf{E} - \left( k' \right)^{2} \mathbf{K} }{ k \left( k' \right)^{2} } 
\end{equation}
 is to be used in a key way in our work. 

 Next, for a positive argument $r$, the \emph{elliptic lambda function} $\lambda^{\ast}(r)$ is defined so that $ \lambda^{\ast}(r):= k_r$, where $0 < 
 k_{r} < 1$ is such that $ \frac{ \mathbf{K}' }{\mathbf{K}}(k_{r}) = \sqrt{r}$ \cite[p.~67, eq.~(3.2.2)]{BorweinBorwein1987text}. The function $ 
 \lambda^{\ast}$ may equivalently be defined through the use of Ramanujan's $G$-function \cite{Ramanujan1914} such that 
\begin{equation}\label{202310067777372787A7M1A}
 \prod_{k=0}^{\infty} \left( 1 + e^{-(2k+1) \pi \sqrt{n}} \right) 
 = 2^{1/4} e^{-\pi \sqrt{n}/24} G_{n}, 
\end{equation}
 and the $G$-function in \eqref{202310067777372787A7M1A}
 is to be heavily involved in our proofs of Aldawoud's conjectures. 

 The \emph{elliptic alpha function} is such that 
 $ \alpha(r) := \frac{ \mathbf{E}' }{\mathbf{K}} - \frac{\pi}{4 \mathbf{K}^2}$ 
 \cite[p.~152, eq.~(5.1.1)]{BorweinBorwein1987text}, 
 under the modulus $k = \lambda^{\ast}(r)$, and we may rewrite this as \cite[p.~152, eq.~(5.1.2)]{BorweinBorwein1987text}: 
\begin{equation}\label{mainalpha}
	\alpha(r) = \frac{\pi}{4 \mathbf{K}^{2}} - \sqrt{r} \left( \frac{\mathbf{E}}{\mathbf{K}} - 1 \right). 
\end{equation}

 We adopt a notational convention from \cite{BaruahBerndt2010} by setting
\begin{align*}
 A_n := \frac{ \left( \frac{1}{2} \right)_{n}^3 }{ \left( 1 \right)_{n}^3},\qquad
 B_n:=\frac{ 
	\left( \frac{1}{4} \right)_{n} \left( \frac{1}{2} \right)_{n} \left( \frac{3}{4} \right)_{n} }{ \left( 1 \right)_{n}^3 },\qquad
 C_n:=\frac{ 
	\left( \frac{1}{6} \right)_{n} \left( \frac{1}{2} \right)_{n} \left( \frac{5}{6} \right)_{n} }{ \left( 1 \right)_{n}^3 }.
\end{align*}
We also set $ D_n:=\frac{ 
 	\left( \frac{1}{3} \right)_{n} \left( \frac{1}{2} \right)_{n} \left( \frac{2}{3} \right)_{n} }{ \left( 1 \right)_{n}^3 }$. 
 A hypergeometric identity of fundamental importance in the study of Ramanujan-type series is such that
\begin{equation*}
 \sum_{n=0}^{\infty} A_{n} x^{n} 
 = \frac{4 \mathbf{K}^2 \left( \sqrt{\frac{1 - \sqrt{1-x}}{2}} \right)}{\pi ^2}. 
\end{equation*}
This relation is highlighted in a different form in \cite[p.~180, Theorem 5.7(a), eq.~(i)]{BorweinBorwein1987text}. The generating function of $B_n$ is of 
 great importance in the quartic theory of elliptic functions, 
 and we recall the following equivalent formulation of \cite[p.~181, Theorem 5.7(b), eq.~(iv)]{BorweinBorwein1987text}:
\begin{equation}\label{Bgenerating}
	\sum_{n=0}^{\infty} B_{n} x^{n} = 
	\frac{4 \sqrt{2}\, \mathbf{K}^2\left(\sqrt{\frac{1}{2} - \frac{\sqrt{\frac{1-\sqrt{1-x}}{x}}}{\sqrt{2}}}\right)}{\pi ^2 \sqrt[4]{2 \sqrt{1-x}+2-x}}. 
\end{equation}
According to \cite[p.~181, Theorem 5.7(c), eq.~(vi)]{BorweinBorwein1987text}, 
\emph{Bailey's cubic transformation} can be stated as
\begin{equation}\label{Baileycubic}
 \sum_{n=0}^{\infty} A_{n} x^{n} = \frac{2}{\sqrt{4-x}} \sum_{n=0}^{\infty} C_{n} \left( \frac{27 x^2}{(4-x)^3} \right)^{n}. 
\end{equation}
Another specialization of Bailey's transformation \cite[p.~181, eq.~(5.5.9)]{BorweinBorwein1987text} tells us that
\begin{equation}\label{Baileycubicanother}
 \sum_{n=0}^{\infty} A_{n} x^{n} = \frac{1}{\sqrt{1-4x}} \sum_{n=0}^{\infty} C_{n} \left( -\frac{27 x}{(1-4x)^3} \right)^{n}. 
\end{equation}
Furthermore, a limiting case of the \emph{Kummer--Goursat transform}, as given by squaring \cite[eq.~(25)]{Vid2009}, provides a connection between the generating functions of $C_n$ and $D_n$: 
\begin{align}
 \sum_{n=0}^{\infty} D_n \big(4 y (1 - y)\big)^{n} = \frac{3}{\sqrt{9 - 8 y}} \sum_{n=0}^{\infty} C_n \left( \frac{ 64 y^3(1-y) }{(9 - 8 y)^3} \right)^{n}. \label{20230219259AM2A} 
\end{align}

\section{New applications of modular relations}\label{sectionmodular}
 Our main identity, which we are to apply in Section \ref{subsectionProofs} to prove conjectures due to Aldawoud \cite{Aldawoud2012}, and which we 
 are to apply in Section \ref{subsectionnewRamanujan} to prove a new Ramanujan-type series of level 3 with a quartic convergence rate and quartic 
 coefficients, is given as Theorem \ref{maintheorem} below. 
 
 The key relation in \eqref{mainobstacle} can be shown to be equivalent to the relation $$ \frac{ (9 - 8 \alpha_{n})^{3} }{ 64 \alpha_{n}^{3}(1 - 
 \alpha_{n}) } = \frac{ ( 4 G_{3n}^{24} - 1 )^3 }{ 27 G_{3n}^{24} } $$ used by the Borwein brothers for evaluating $\alpha_{n}$ and later 
 considered by Chan and Liaw in \cite{ChanLiaw2000}. Our formula for the $a$-value given in Theorem \ref{maintheorem} below appears to be new, 
 and this is to be applied to solve open problems given by Aldawoud in \cite{Aldawoud2012}. 

\begin{theorem}\label{maintheorem}
 We set 
\begin{equation}\label{xinmain}
 x = 4 \left( \left(\lambda^{\ast}(r)\right)^2 - \left(\lambda^{\ast}(r)\right)^4\right). 
\end{equation}
 We further set $y$ as the unique solution in $(-\frac{1}{2},0)$, if such a solution exists, such that 
\begin{equation}\label{mainobstacle}
 \frac{64 (y-1) y^3}{(8 y-9)^3} = -\frac{27 x}{(1-4 x)^3}.
\end{equation} 
 Then the following Ramanujan-type series holds true: 
\begin{equation}\label{2999909293919090969390989A9M91A}
 \frac{1}{\pi} = \sum_{n = 
 0}^{\infty} \frac{ \left( \frac{1}{3} 
 \right)_{n} \left( \frac{1}{2} \right)_{n} \left( \frac{2}{3} \right)_{n} }{ \left( 1 \right)_{n}^3} z^{n} (a+bn), 
\end{equation}
 if we choose 
 \begin{align}
 z &=4 y(1-y), \nonumber \\ 
 a &= \frac{\sqrt{9-8 y}} {6 \sqrt{(1 - 
 4x)^{3}} \big(27 - 36 y+8 y^2\big)}\cdot \Big[ 2 \big(1-4x\big) \big(27 - 36 y + 
 8 y^2\big) \cdot\alpha (r) \nonumber \\ 
 &\quad+ 4x\Big(\big(27 - 36 y+8 y^2\big) -\big(16y-16y^2\big)\sqrt{1 - 
 x}\Big)\cdot \sqrt{r} \nonumber \\ 
 &\quad- \Big(\big(27 - 36 y + 8 y^2\big)-\big(27 - 44 y + 16 y^2\big)\sqrt{1-x}\Big)\cdot \sqrt{r}\Big], \nonumber \\ 
 b & = \frac{1}{3} \sqrt{r} (1-2 y), \label{20231006311000000000000000000005AM1A} 
 \end{align}
 provided that $|z|<1$ (cf.\ \cite[pp.\ 185--186]{BorweinBorwein1987text}) . 
 \end{theorem}

\begin{proof}
 For $x$ as in \eqref{xinmain}, we may obtain, from \eqref{Baileycubicanother}, that 
 $$ \sum_{n=0}^{\infty} C_{n} \left( -\frac{27x}{(1-4x)^3} \right)^{n} = \frac{ 4 
 \text{{\bf K}}^{2}\left( \lambda^{\ast}(r) \right) }{\pi^2} \sqrt{1-4x} $$
 and, via \eqref{mainODE}, that 
\begin{align*}
 & \sum_{n=0}^{\infty} 
 C_{n} \left( -\frac{27x}{(1-4x)^3} \right)^{n} n = 
 \frac{ \text{{\bf K}}( 
 \lambda^{\ast}(r) }{\pi^2}
 \frac{2 \sqrt{1-4 x} 
 ) }{ \sqrt{1-x} (8 x+1)} \\
 & \big( (2-8 x) \text{{\bf E}}( \lambda^{\ast} (r)) - 
 \left(-4 x+\sqrt{1-x}+1\right) \text{{\bf K}}( \lambda^{\ast} (r)) \big). 
\end{align*}
 For $y$ satisfying \eqref{mainobstacle} in the specified range, we may obtain, from 
 \eqref{20230219259AM2A}, that 
\begin{equation}\label{202307216217A7M71A}
 \sum_{n=0}^{\infty} 
 D_{n} \left( 4(1-y)y \right)^{n} 
 =
 \frac{12 \sqrt{1-4 x}}{ \sqrt{9-8 y}} 
 \frac{ \text{{\bf K}}^{2}\left( \lambda^{\ast}(r) \right)}{\pi^2},
\end{equation} 
 and, via \eqref{mainODE}, that 
\begin{align*}
 & \sum_{n=0}^{\infty}
 (4(1-y)y)^{n} D_{n} n = \frac{48 \sqrt{1-4 x} (y-1) y}{ (9-8 y)^{3/2} (2 y-1)} 
 \frac{ \text{{\bf K}}^{2}\left( 
 \lambda^{\ast}(r)\right)\textbf{}}{\pi^2} + \\
 & \frac{6 \sqrt{1-4 x} (4 y (2 y-9)+27)}{ \sqrt{1-x} (8 x+1) (9-8 y)^{3/2} (2 y-1)} \frac{\text{{\bf K}}( \lambda^{\ast}(r))}{\pi^2} \times \\
 & \left((8 x-2) \text{{\bf E}}(\lambda^{\ast}(r))+\left(-4 x+\sqrt{1 
 -x}+1\right) \text{{\bf K}}(\lambda^{\ast}(r))\right).
\end{align*}
  By rewriting $\text{{\bf E}}\left(   \lambda^{\ast}(r) \right)$  according to the elliptic   alpha function identity in  \eqref{mainalpha},   we may express  
  $\sum_{n=0}^{\infty}  (4(1-y)y)^{n} D_{n} n$   as a linear combination of   an algebraic multiple of $\frac{1}{\pi}$   and an algebraic multiple of   $  
 \frac{\text{{\bf K}}^{2}\left( \lambda^{\ast}(r) \right)}{\pi^2}$.   The desired result can then be shown to follow by taking an appropriate linear  
  combination of    $ \sum_{n=0}^{\infty}   (4(1-y)y)^{n} D_{n} n $   and the left-hand side of  \eqref{202307216217A7M71A}. More explicitly, we would 
 obtain that the  $b$-value in \eqref{2999909293919090969390989A9M91A}, for $x$ and $y$ 
 as specified in \eqref{xinmain} and \eqref{mainobstacle}, is such that 
\begin{equation}\label{20231006316AM1111111111111111A}
 b = \frac{(1+8x) \sqrt{1-x}\cdot (1-2 y) \sqrt{(9-8 y)^{3}} \cdot \sqrt{r}}{3 \sqrt{(1-4x)^{3}} \big(27 - 36 y+8 y^2\big)}, 
\end{equation}
 with the right-hand side of \eqref{20231006316AM1111111111111111A}
 reducing to the specified value in \eqref{20231006311000000000000000000005AM1A}. 
\end{proof}

\subsection{Proofs of Aldawoud's conjectures}\label{subsectionProofs}
 Aldawoud \cite[p.\ 14]{Aldawoud2012} 
 provided the formula 
 $$ \sqrt{1-108x} \sum_{k=0}^{\infty} \frac{ (3k)! (2k)! }{ \left( k! \right)^{5} } 
 \left( k + \lambda \right) x^{k} = \frac{\sqrt{3/N}}{\pi} $$ 
 as a special case 
 of a formula given as 
\begin{equation}\label{20231003777273727A7M1A}
 \sqrt{1-4ax} \sum_{k=0}^{\infty} \binom{2k}{k} s(k) (k + \lambda) x^{k} = \frac{1}{\rho} 
\end{equation}
 in \cite[p.\ 10]{Aldawoud2012}, referring to 
 \cite{Aldawoud2012} for details. 
 As stated in \cite[p.\ 10]{Aldawoud2012}, 
 the evaluation of the parameter $\lambda$ in 
 \eqref{20231003777273727A7M1A} depends on a result 
 due to Chan, Chan, and Lui \cite{ChanChanLiu2004} whereby 
\begin{equation}\label{CCLderivative}
 \lambda = \frac{x}{2N} \frac{dM}{dx} \Big|_{q = e^{-2\pi/\sqrt{N \ell}}}, 
\end{equation}
 where 
\begin{equation}\label{20231003302A77777777777777777777M772A}
 M(q) = \frac{Z(q)}{Z\left( q^{N} \right) } 
\end{equation}
 and 
\begin{equation}\label{ZqAldawoud}
 Z(q) = \sum_{k=0}^{\infty} h(k) x^{k}(q), 
\end{equation}
    referring to \cite{Aldawoud2012,ChanChanLiu2004} for details.      However, actually finding and proving closed forms for the required     values in  
    \eqref{CCLderivative}--\eqref{ZqAldawoud}     is, typically, of a very challenging nature.      As described in \cite[\S3]{Aldawoud2012},     conjectured values   
    for $x$ and $\lambda$ were obtained experimentally, via numerical experiments with Maple.     In particular, the Maple {\tt identify} command was applied in  
    \cite[\S3]{Aldawoud2012}      to obtain conjectured closed forms.      However, proofs for the experimentally discovered values for $x$ and $\lambda$      are  
    not given in \cite{Aldawoud2012},      and proving these conjectured valuations is very difficult.  

 A conjectured Ramanujan-type series given in Table 3.7 in \cite[p.\ 28]{Aldawoud2012} is such that 
\begin{equation}\label{input19}
 \sqrt{1 - 108x} \sum_{k=0}^{\infty} \frac{ \left( \frac{1}{3} \right)_{k} 
 \left( \frac{1}{2} \right)_{k} \left( \frac{2}{3} \right)_{k} }{ \left( k! \right)^{3} } 
 (k + \lambda) (108 x)^{k} 
 = \frac{\sqrt{3/N}}{\pi}, 
\end{equation}
 where 
\begin{align}
 & N = 19, \label{Nequals19} \\
 & x = -\frac{4261}{22781250}-\frac{3913}{91125000} \sqrt{19}, \\ 
 & \lambda = \frac{49}{510} + \frac{73}{9690} \sqrt{19}. \label{twoafterN19} 
\end{align}
 As below, we are to prove that the conjectured Ramanujan-type series in \eqref{input19}, for the specified values in 
 \eqref{Nequals19}--\eqref{twoafterN19}, agrees with the $r = 57$ case of Theorem \ref{maintheorem}. 

 A series for $\frac{1}{\pi}$ is said to be \emph{quadratic-irrational} if an algebraic number times this series yields an infinite sum of quadratic irrational 
 numbers \cite[p.\ 7]{Aldawoud2012}. 
 Given the huge amount of interest in \emph{rational} Ramanujan-type series, 
 which refer to Ramanujan-type series for $\frac{1}{\pi}$ that yield a rational expansion 
 after being multiplied by an algebraic value, this motivates our proofs of 
 Aldawoud's conjectured and experimentally discovered results on quadratic-irrational Ramanujan-type series, 
 which may be seen as a natural ``next step'' after rational Ramanujan-type series. 

\begin{theorem}
 Aldawoud's conjectured Ramanujan-type series formula \cite[p.\ 28]{Aldawoud2012} 
\begin{align}
\begin{split}
 \frac{1}{\pi}
 = & \sum_{n = 0}^{\infty} \frac{ \left( \frac{1}{3} \right)_{n} 
 \left( \frac{1}{2} \right)_{n} \left( \frac{2}{3} \right)_{n} }{ \left( 1 \right)_{n}^{3} } 
 \Bigg( \frac{5719+13 \sqrt{19}}{2250} n + \frac{1654+133 \sqrt{19}}{6750} \Bigg) \times \\ 
 & \ \ \ \ \ \left( \frac{-17044-3913 \sqrt{19}}{843750} \right)^{n} 
\end{split}\label{202310011111111110100000485707PM1A}
 \end{align}
 holds true. 
\end{theorem}

\begin{proof}
 Our strategy is to make use of the $r = 57$ case of Theorem \ref{maintheorem} together with 
 closed forms for $\lambda^{\ast}(57)$ and $ 
 \alpha(57)$. In this regard, 
 the explicit evaluation $$ G_{57}^{-6} = \left( \frac{ 3 \sqrt{19} - 13 }{\sqrt{2}} \right) \left( \frac{\sqrt{3}-1}{\sqrt{2}} \right)^{3} 
 $$ is given by Borwein and Borwein in \cite[p.\ 148]{BorweinBorwein1987text}. We thus obtain that 
\begin{equation}\label{202302221202024282432AAA}
 \lambda^{\ast}\left( 57 \right) 
 = \frac{1}{2} \left( \sqrt{1 + G_{57}^{-12}} - \sqrt{1 - G_{57}^{-12}} \right), 
\end{equation}
 according to an equivalent definition for the elliptic lambda function \cite[p.\ 161]{BorweinBorwein1987text}. 
 According to Theorem 5.3 from Borwein and Borwein's text \cite[p.\ 158]{BorweinBorwein1987text}, 
 we have that 
\begin{equation}\label{20777777273717070747874727PM1A}
 \alpha(p) = \sqrt{p} 
 \frac{1 + \left( \lambda^{\ast}(p) \right)^{2}}{3} 
 - \frac{\sigma(p)}{6}, 
\end{equation}
 for $\sigma(p) := R_{p}(k', k) $ 
 as defined in \cite[\S5.2]{BorweinBorwein1987text}, 
 with $k := e^{-\pi \sqrt{p}}$. 
 The formula 
 $$\sigma(57) = 3 G_{57}^{-6} \sqrt{ \frac{2 \sqrt{19} + 5 \sqrt{3} }{2} } 
 \left( 5 \sqrt{57} + 13 \sqrt{19} + 49 \sqrt{3} + 19 \right) $$
 is given in the Borwein brothers' text \cite[p.\ 167]{BorweinBorwein1987text}, 
 so that the closed form obtained from \eqref{202302221202024282432AAA}, 
 in conjunction with the 
 $p = 57$ case of \eqref{20777777273717070747874727PM1A}, 
 give us a closed form for $\alpha(57)$. 
 As in Theorem 
 \ref{maintheorem}, 
 we write $ x = 4 \big( \big(\lambda^{\ast}(r)\big)^2 - \big(\lambda^{\ast}(r)\big)^4\big)$, and we may verify that the value 
\begin{equation}\label{20231006332A11111111222222222222aaAA2M22}
 y = \frac{1}{750} \left(375-\sqrt{\frac{1}{6} \left(860794+3913 \sqrt{19}\right)}\right) 
\end{equation}
 is such that \eqref{mainobstacle} holds. 
 It is thus a matter of routine 
 to demonstrate how 
 the values of $z$, $a$, and $b$ 
 in Theorem \ref{maintheorem}, 
 subject to the closed forms for $\lambda^{\ast}(57)$ and $\alpha(57)$ and $y$, 
 reduce to the corresponding values in \eqref{202310011111111110100000485707PM1A}. 
\end{proof}

\begin{remark}
 As a byproduct of our method, the $y$-values 
 involved in our proofs, as in \eqref{20231006332A11111111222222222222aaAA2M22}, 
 give us new, closed-form evaluations for the 
 cubic singular modulus function defined in \eqref{cubicalpha}. 
\end{remark}

 The fourth convergent series conjectured by Aldawoud and listed in Table 3.7 in 
 \cite[p.\ 28]{Aldawoud2012} is given in an equivalent form in 
 the below Theorem. 

 \begin{theorem}
 Aldawoud's conjectured Ramanujan-type series formula \cite[p.\ 28]{Aldawoud2012} 
\begin{align*}
 \frac{1}{\pi}
 = & \sum_{n=0}^{\infty} \frac{ \left( \frac{1}{3} \right)_{n} 
 \left( \frac{1}{2} \right)_{n} \left( \frac{2}{3} \right)_{n} }{ \left( 1 \right)_{n}^{3} } 
 \left( \frac{217 + 35113 \sqrt{31}}{60750} n 
 + \frac{14662 + 7843 \sqrt{31} }{182250} \right) \times \\
 & \ \ \ \ \ \left( \frac{-1368394-245791 \sqrt{31}}{615093750} \right)^{n}
\end{align*}
 holds true. 
\end{theorem}

\begin{proof}
 Our strategy is to make use of the $r = 93$ case of 
 Theorem \ref{maintheorem}. 
 The explicit evaluation 
\begin{equation}\label{G93power}
 G_{93}^{-6} = \left( \frac{39 - 7 \sqrt{31}}{\sqrt{2}} \right) 
 \left( \frac{\sqrt{31} - 3 \sqrt{3}}{2} \right)^{3/2} 
\end{equation}
 is given in Borwein and Borwein's text \cite[p.\ 148]{BorweinBorwein1987text}, 
 and the explicit evaluation 
\begin{equation}\label{207237177007787170797A7c7o7mm7a77M71A}
 \sigma(93) = 6 G_{93}^{-6} \left( \frac{\sqrt{3} + 1}{2} \right)^{3} 
 (15 \sqrt{93} + 13 \sqrt{31} + 201 \sqrt{3} + 217) 
\end{equation}
 is given in the same text \cite[p.\ 168]{BorweinBorwein1987text}. 
 The closed form in \eqref{G93power} leads us to the closed form 
\begin{equation}\label{20727371707078771718781878A878M81A}
 \lambda^{\ast}\left( 93 \right) 
 = \frac{1}{2} \left( \sqrt{1 + G_{93}^{-12}} - \sqrt{1 - G_{93}^{-12}} \right), 
\end{equation}
 and the closed form in \eqref{207237177007787170797A7c7o7mm7a77M71A} 
 gives us, according to the $p = 93$
 case of \eqref{20777777273717070747874727PM1A}, a closed form for $\alpha(93)$. 
 The closed form in \eqref{20727371707078771718781878A878M81A} leads us to a closed form 
 for $ x = 4 \left( \left(\lambda^{\ast}(93)\right)^2 - \left(\lambda^{\ast}(93)\right)^4\right)$, 
 and we may verify that $x$ and 
 $$ y = \frac{-1368394-245791 \sqrt{31}}{121500 \left(10125+\sqrt{\frac{1}{6} \left(616462144+245791 \sqrt{31}\right)}\right)} $$ 
 satisfy the desired relation in \eqref{mainobstacle}, with $y$ in the specified interval. 
 So, from the specified values for $r$, $x$, $y$, and $\alpha(r)$, it is a matter of routine 
 to verify, with the use of Theorem \ref{maintheorem}, 
 our closed forms for $z$, $a$, and $b$. 
\end{proof}

 The below proof of one of Aldawoud's conjectures from \cite{Aldawoud2012} is especially remarkable 
 due to the recursive approach used in this proof and due to the 
 extremely nontrivial computations required and due to its ``computer proof'' nature. 

\begin{theorem}
 Aldawoud's conjectured Ramanujan-type series formula \cite[p.\ 28]{Aldawoud2012} 
\begin{align}
\begin{split}
 \frac{1}{\pi}
 = & \sum_{n=0}^{\infty} \frac{ \left( \frac{1}{3} \right)_{n} 
 \left( \frac{1}{2} \right)_{n} \left( \frac{2}{3} \right)_{n} }{ \left( 1 \right)_{n}^{3} } 
 \left( \frac{297 \sqrt{3}-91 \sqrt{11}}{64} n + \frac{75 \sqrt{3}-33 \sqrt{11}}{64} \right) \times \\
 & \ \ \ \ \ \left( \frac{ 2457 \sqrt{33} - 14121 }{2048} \right)^{n}
\end{split}\label{2023180808888175727A7M1A}
\end{align}
 holds true. 
\end{theorem}

\begin{proof}

 Our strategy is to use the $r = 99$ case of Theorem \ref{maintheorem}. 
 Let 
 \begin{align}\label{eq:rho-value} 
		\rho=\tfrac{1}{6}\big({\textstyle{-\sqrt[3]{16}}+\sqrt[3]{38-6\sqrt{33}}+\sqrt[3]{38+6\sqrt{33}}}\big)
	\end{align}
 be the unique real solution to the cubic equation
	\begin{align}\label{eq:rho-cubic}
		2\rho^3+2^{\frac{4}{3}}\rho^2-1=0.
	\end{align}
 Then
 \begin{align}
		\lambda^*(11) &= \sqrt{\tfrac{1-\sqrt{1-4\rho^{12}}}{2}},\label{eq:lambda-11}\\
		\alpha(11) &= \tfrac{\left(3-\sqrt{1-4\rho^{12}}\right)\sqrt{11}}{6} - \tfrac{3\rho^6+2\rho^3+2}{3}.\label{eq:alpha-11}
	\end{align} 
 The value for $\lambda^{\ast}(11)$ 
 is given, in an equivalent form in \cite[p.\ 162]{BorweinBorwein1987text}. 
 The value for $\alpha(11)$ is given in an equivalent way 
 in \cite[p.\ 158]{BorweinBorwein1987text}. 
 According to the recursion 
\begin{equation}\label{GntoG9n}
 9 = \left( 1 + 2 \sqrt{2} \frac{ G_{9n}^{3} }{ G_{n}^{9} } \right) 
 \left( 1 + 2 \sqrt{2} \frac{ G_{n}^{3} }{ G_{9n}^{9} } \right) 
\end{equation}
 given in \cite[p.\ 145]{BorweinBorwein1987text} 
 together with the recursion 
 $$ \alpha(9r) = s^2(r) \alpha(r) 
 - \frac{ \sqrt{r} (s^2(r) + 2 s(r) - 3) }{2} $$ 
 given in \cite[p.\ 160]{BorweinBorwein1987text}, where 
 $$ s(r) := \sqrt{1 + 4 \frac{ (k k')^{3/4} }{ (ll')^{1/4} }}, $$
 and where $l:= \lambda^{\ast}(r)$ and $k:= \lambda^{\ast}(9r)$, 
 we may derive closed forms for $\lambda^{\ast}(99)$ and $\alpha(99)$, 
 as below, letting 
 $$ G_{11} = \frac{1}{\sqrt[12]{2} \sqrt[24]{\frac{1}{2} \left(1-\sqrt{1-4 \rho ^{12}}\right)-\frac{1}{4} \left(1-\sqrt{1-4 \rho ^{12}}\right)^2}}. $$

 Inputting 
\begin{verbatim}
Solve[(64 (y - 1) y^3)/(8 y - 9)^3 == -((27 x)/(1 - 4 x)^3) , y]
\end{verbatim}
 into Mathematica, we may verify that the corresponding output exhausts the possible ways of expressing $y$ in terms of $x$, subject to the desired 
 relation in \eqref{mainobstacle}.  Out of all of the possible ways of expressing $y$ in terms of $x$,  we may verify that: If  $ x = 4 \big( \big( 
 \lambda^{\ast}(99) \big)^{2} - \big( \lambda^{\ast}(99) \big)^{4} \big)$, then it follows that $y$  is equal to the value denoted as ``{\tt y}'' in the 
 Mathematica input shown below.  We have that $G_{99}$, according to the recursion in \eqref{GntoG9n}, 
 is equal to the expression denoted as ``{\tt G99}'' in the Mathematica input shown below. 
 This allows us to determine a closed form for 
\begin{equation}\label{202302221202024282qqq4qqq3q2qAqAA}
 \lambda^{\ast}\left( 99 \right) 
 = \frac{1}{2} \left( \sqrt{1 + G_{99}^{-12}} - \sqrt{1 - G_{99}^{-12}} \right). 
\end{equation}
 Our strategy, at this point, is to show that $4y(1-y)$, for the value of $y$ specified above, is equal 
 to the convergence rate in \eqref{2023180808888175727A7M1A}. This, would give 
 us an explicit closed form for the $z$-value involved in Theorem 
 \ref{maintheorem}, for the $r = 99$ case. 
 So, to show that 
\begin{equation}\label{strategyPolynomial}
 4 y(1-y) - \frac{2457 \sqrt{33} - 14121}{2048} 
\end{equation}
 vanishes, our strategy is to use the {\tt MinimalPolynomial} command in the Wolfram language, i.e., 
 to show that the minimal polynomial for \eqref{strategyPolynomial} is given
 by a polynomial of degree $1$ with $0$ as its constant term. 

\begin{verbatim}
\[Rho] = (-2^(4/3) + (38 - 6*Sqrt[33])^(1/3) + (38 + 6*Sqrt[33])^
(1/3))/6;

G11 = 1/(2^(1/12)*(Sqrt[(1 - Sqrt[1 - 4*\[Rho]^12])/2]^2 - 
Sqrt[(1 - Sqrt[1 - 4*\[Rho]^12])/2]^4)^(1/24));

G99 = (G11^9/Sqrt[2] + Sqrt[G11^2 + G11^10 + G11^18]/Sqrt[2] + 
Sqrt[-G11^2 - G11^10 + 2*G11^18 - (2*G11^3)/Sqrt[G11^2 + G11^10 + 
G11^18] + (2*G11^27)/Sqrt[G11^2 + G11^10 + G11^18]]/Sqrt[2])^(1/3);

lambda99 = (1/2)*(Sqrt[1 + G99^(-12)] - Sqrt[1 - G99^(-12)]);

x = 4*(lambda99^2 - lambda99^4);

y = -((1 - 228*x + 48*x^2 - 64*x^3)/(4*(-1 + 4*x)^3)) + (1/2)*
Sqrt[-((729*x)/(-1 + 4*x)^3) + (1 - 228*x + 48*x^2 - 64*x^3)^2/(4*
(-1 + 4*x)^6) + (243*x)/(-1 + 12*x - 48*x^2 + 64*x^3) + (81*(-x + 
12*x^2 - 48*x^3 + 64*x^4))/(2*2^(2/3)*(-1 + 4*x)^3*(x - 51*x^2 + 
564*x^3 - 2576*x^4 + 5568*x^5 - 6144*x^6 + 4096*x^7 + Sqrt[x^2 + 
6*x^3 - 159*x^4 - 472*x^5 + 11376*x^6 - 2304*x^7 - 383232*x^8 + 
1124352*x^9 + 3575808*x^10 - 25821184*x^11 + 55050240*x^12 - 
50331648*x^13 + 16777216*x^14])^(1/3)) + (1/(4*2^(1/3)*(-1 + 4*x)^
3))*27*(x - 51*x^2 + 564*x^3 - 2576*x^4 + 5568*x^5 - 6144*x^6 + 
4096*x^7 + Sqrt[x^2 + 6*x^3 - 159*x^4 - 472*x^5 + 11376*x^6 - 
2304*x^7 - 383232*x^8 + 1124352*x^9 + 3575808*x^10 - 25821184*x^
11 + 55050240*x^12 - 50331648*x^13 + 16777216*x^14])^(1/3)] - (1/
2)*Sqrt[-((729*x)/(-1 + 4*x)^3) + (1 - 228*x + 48*x^2 - 64*x^3)^2/
(2*(-1 + 4*x)^6) - (243*x)/(-1 + 12*x - 48*x^2 + 64*x^3) - (81*
(-x + 12*x^2 - 48*x^3 + 64*x^4))/(2*2^(2/3)*(-1 + 4*x)^3*(x - 51*
x^2 + 564*x^3 - 2576*x^4 + 5568*x^5 - 6144*x^6 + 4096*x^7 + 
Sqrt[x^2 + 6*x^3 - 159*x^4 - 472*x^5 + 11376*x^6 - 2304*x^7 - 
383232*x^8 + 1124352*x^9 + 3575808*x^10 - 25821184*x^11 + 
55050240*x^12 - 50331648*x^13 + 16777216*x^14])^(1/3)) - (1/(4*2^
(1/3)*(-1 + 4*x)^3))*27*(x - 51*x^2 + 564*x^3 - 2576*x^4 + 5568*x^
5 - 6144*x^6 + 4096*x^7 + Sqrt[x^2 + 6*x^3 - 159*x^4 - 472*x^5 + 
11376*x^6 - 2304*x^7 - 383232*x^8 + 1124352*x^9 + 3575808*x^10 - 
25821184*x^11 + 55050240*x^12 - 50331648*x^13 + 16777216*x^14])^
(1/3) + ((6561*x)/(-1 + 4*x)^3 + (2916*x*(1 - 228*x + 48*x^2 - 
64*x^3))/(-1 + 4*x)^6 - (1 - 228*x + 48*x^2 - 64*x^3)^3/(-1 + 4*
x)^9)/(4*Sqrt[-((729*x)/(-1 + 4*x)^3) + (1 - 228*x + 48*x^2 - 64*
x^3)^2/(4*(-1 + 4*x)^6) + (243*x)/(-1 + 12*x - 48*x^2 + 64*x^3) + 
(81*(-x + 12*x^2 - 48*x^3 + 64*x^4))/(2*2^(2/3)*(-1 + 4*x)^3*(x - 
51*x^2 + 564*x^3 - 2576*x^4 + 5568*x^5 - 6144*x^6 + 4096*x^7 + 
Sqrt[x^2 + 6*x^3 - 159*x^4 - 472*x^5 + 11376*x^6 - 2304*x^7 - 
383232*x^8 + 1124352*x^9 + 3575808*x^10 - 25821184*x^11 + 
55050240*x^12 - 50331648*x^13 + 16777216*x^14])^(1/3)) + (1/(4*2^
(1/3)*(-1 + 4*x)^3))*27*(x - 51*x^2 + 564*x^3 - 2576*x^4 + 5568*x^
5 - 6144*x^6 + 4096*x^7 + Sqrt[x^2 + 6*x^3 - 159*x^4 - 472*x^5 + 
11376*x^6 - 2304*x^7 - 383232*x^8 + 1124352*x^9 + 3575808*x^10 - 
25821184*x^11 + 55050240*x^12 - 50331648*x^13 + 16777216*x^14])^
(1/3)])]

MinimalPolynomial[4*y*(1 - y) - (2457*Sqrt[33] - 14121)/2048, w]
\end{verbatim}

 The output gives us that it is a matter of routine to verify the desired vanishing of 
 \eqref{strategyPolynomial}. This gives us that 
\begin{equation}\label{yafterstrategy}
 y = \frac{1}{128} \left(155-27 \sqrt{33}\right). 
\end{equation}
 Resetting the inputted value for {\tt y} in Mathematica as 
\begin{verbatim}
y = (1/128)*(155 - 27*Sqrt[33])
\end{verbatim}
 and inputting
\begin{verbatim}
r = 99;

lambda11 = Sqrt[(1 - Sqrt[1 - 4*\[Rho]^12])/2];

s11 = Sqrt[1 + 4*((lambda99*Sqrt[1 - lambda99^2])^(3/4)/(lambda11*
Sqrt[1 - lambda11^2])^(1/4))];

alpha11 = ((3 - Sqrt[1 - 4*\[Rho]^12])*Sqrt[11])/6 - (3*\[Rho]^6 + 
2*\[Rho]^3 + 2)/3;

alpha99 = s11^2*alpha11 - (Sqrt[11]*(s11^2 + 2*s11 - 3))/2;

a = (Sqrt[9 - 8*y]/(6*Sqrt[(1 - 4*x)^3]*(27 - 36*y + 8*y^2)))*(2*
(1 - 4*x)*(27 - 36*y + 8*y^2)*alpha99 + 4*x*((27 - 36*y + 8*y^
2) - (16*y - 16*y^2)*Sqrt[1 - x])*Sqrt[r] - ((27 - 36*y + 8*y^2) - 
(27 - 44*y + 16*y^2)*Sqrt[1 - x])*Sqrt[r]);

MinimalPolynomial[a - (75*Sqrt[3] - 33*Sqrt[11])/64, w]
\end{verbatim}
\noindent into Mathematica, we find that it is a matter of routine to verify that 
 the $a$-value in Theorem \ref{maintheorem} 
 reduces so that 
 $ a = \frac{1}{64} \left(75 \sqrt{3}-33 \sqrt{11}\right)$, 
 as desired. From the closed form for $y$ in \eqref{yafterstrategy}, 
 it is immediate that the desired $b$-value holds, 
 according to the relation $ b = \frac{1}{3} \sqrt{r} (1-2 y)$ 
 specified in Theorem \ref{maintheorem}. 
\end{proof}

\subsection{A new Ramanujan-type series}\label{subsectionnewRamanujan}
 To the best of our knowledge, the only previously published and previously proved 
 Ramanujan-type series of level 3 with negative convergence rates are listed below, 
 apart from Aldawoud's conjectures given in \cite{Aldawoud2012}. 
 The Ramanujan-type series among 
 \eqref{9290929390929199 727PM1A}--\eqref{20230302110111112111111P1M1A} are due to 
 Chan, Liaw, $\&$ Tan \cite{ChanLiawTan2001}, 
 and the Ramanujan-type series in 
 \eqref{202303021108P200addM1A}--\eqref{lastknown} is due to 
 Berkovich, Chan, $\&$ Schlosser \cite{BerkovichChanSchlosser2018}. 
 Each of the following results can be shown to be a special case of Theorem \ref{maintheorem}. 
\begin{align}
 & \sum_{n=0}^{\infty} \left( -\frac{9}{16} \right)^{n}
 \frac{ \left( \frac{1}{2} \right)_{n} \left( \frac{1}{3} \right)_{n} \left( \frac{2}{3} \right)_{n} }{ (n!)^3 }
 (5n+1) = \frac{4}{\pi \sqrt{3}}, \label{9290929390929199 727PM1A} \\ 
 & \sum_{n=0}^{\infty} \left( -\frac{1}{16} \right)^{n} \frac{ \left( \frac{1}{2} \right)_{n} 
 \left( \frac{1}{3} \right)_{n} 
 \left( \frac{2}{3} \right)_{n} }{(n!)^{3}} (51n+7) = \frac{12 \sqrt{3}}{\pi}, \label{202303021101001001007PM1A} \\ 
 & \sum_{n=0}^{\infty} \left( -\frac{1}{80} \right)^{n} \frac{ \left( \frac{1}{2} \right)_{n}
 \left( \frac{1}{3} \right)_{n} 
 \left( \frac{2}{3} \right)_{n} }{(n!)^3} (9n+1) = \frac{4 \sqrt{3}}{\pi \sqrt{5}}, \\ 
 & \sum_{n=0}^{\infty} \left( -\frac{1}{1024} \right)^{n} \frac{ \left( \frac{1}{2} \right)_{n} 
 \left( \frac{1}{3} \right)_{n} \left( \frac{2}{3} \right)_{n} }{(n!)^3} (615 n + 53) = \frac{96 \sqrt{3}}{\pi}, \\ 
 & \sum_{n=0}^{\infty} \left( -\frac{1}{3024} \right)^{n}
 \frac{ \left( \frac{1}{2} \right)_{n} 
 \left( \frac{1}{3} \right)_{n} 
 \left( \frac{2}{3} \right)_{n} }{(n!)^3} (165n+13) = \frac{2^2 \cdot 3^3}{\pi \sqrt{7}}, \\
 & \sum_{n=0}^{\infty} \frac{(-1)^n}{500^{2n}} \frac{ \left( \frac{1}{2} \right)_{n} 
 \left( \frac{1}{3} \right)_{n} \left( \frac{2}{3} \right)_{n} }{(n!)^{3}} (14151n + 827) 
 = \frac{1500\sqrt{3}}{\pi}, \label{20230302110111112111111P1M1A} \\
 & \sum_{n=0}^{\infty} \left( -\frac{194}{1331} - \frac{225 \sqrt{3}}{2662} \right)^{n} \frac{ 
 \left( \frac{1}{2} \right)_{n} \left( \frac{1}{3} \right)_{n} 
 \left( \frac{2}{3} \right)_{n} }{(n!)^3} \cdot \label{202303021108P200addM1A} \\
 & \left( \left( \frac{45 \sqrt{3}}{22} + \frac{5}{22} \right) n 
 + \frac{13}{66} + \frac{3 \sqrt{3}}{11} \right) = \frac{\sqrt{3}}{\pi}. \label{lastknown} 
\end{align}

 The Berkovich--Chan--Schlosser formula reproduced in 
 \eqref{202303021108P200addM1A}--\eqref{lastknown} \cite{BerkovichChanSchlosser2018} 
 is the same as the Ramanujan-type series corresponding to the $N = 11$
 case of Table 3.7 in \cite[p.\ 28]{Aldawoud2012}. 
 So, the proof via Wronskians of theta functions 
 of \eqref{202303021108P200addM1A}--\eqref{lastknown} 
 due to Berkovich et al.\ \cite{BerkovichChanSchlosser2018} 
 actually solves an open problem from 
 2012 \cite{Aldawoud2012}, 
 but the authors of \cite{BerkovichChanSchlosser2018} 
 were unaware of this. The same formula in 
 \eqref{202303021108P200addM1A}--\eqref{lastknown} 
 can also be proved 
 using our main identity given in  Theorem \ref{maintheorem}, 
 according to the $r = 33$  case, and with the use of known 
 values for $\lambda^{\ast}(33)$ and $\alpha(33)$
 that may be obtained from the \textit{Pi and the {AGM}} 
  text \cite{BorweinBorwein1987text}.  
  This has led us to consider the $r =  30$ case of  
  Theorem \ref{maintheorem},  as below.   It appears that the Ramanujan-type series highlighted in 
  Theorem   \eqref{newspecialcasetheorem} below is new. 

\begin{theorem}\label{newspecialcasetheorem}
 The Ramanujan-type series formula 
\begin{align*}
 \frac{1}{\pi} 
  = & \sum_{n=0}^{\infty} 
   \frac{ \left( \frac{1}{3} \right)_{n} 
  \left( \frac{1}{2} \right)_{n} \left( \frac{2}{3} \right)_{n} }{ \left(  1  \right)_{n}^{3} } 
 \left( \frac{22617 \sqrt{15}+27314 \sqrt{10}-36036 \sqrt{6}-87049}{3375} \right)^{n} \times \\
 & \Bigg( \frac{54 \sqrt{15}+68 \sqrt{10}-72 \sqrt{6}-193}{135} + 
 \frac{63 \sqrt{15}+71 \sqrt{10}-63 \sqrt{6}-217}{45} n \Bigg). 
\end{align*}
 holds true. 
\end{theorem}

\begin{proof}
 This can be shown to follow in a direct way from Theorem \ref{maintheorem}, 
 for the $r = 30$ case, with the use of known closed forms
 such that 
 $$ \lambda^{\ast}(30) = \left( \sqrt{3} - \sqrt{2} \right)^{2} \left( 2 - \sqrt{3} \right) 
 \left( \sqrt{6} - \sqrt{5} \right) \left( 4 - \sqrt{15} \right) $$
 and 
\begin{align*}
 \alpha\left( 30 \right) 
 = & \frac{1}{2} \Big\{ \sqrt{30} - \left( 2 + \sqrt{5} \right)^{2} \left( 3 + \sqrt{10} \right) ^{2} \times \\ 
 & \left( - 6 - 5 \sqrt{2} - 3 \sqrt{5} - 2 \sqrt{10} 
 + \sqrt{6} \sqrt{57 + 40 \sqrt{2}} \right) \times \\
 & \left[ 56 + 38 \sqrt{2} + \sqrt{30} \left( 2 + \sqrt{5} \right) \left( 3 + \sqrt{10} \right) \right] \Big\}.
\end{align*}
 Using these closed forms, it is a matter of routine to show that the unique $y$-value 
 in Theorem \ref{maintheorem} is such that 
 $$ y = \frac{1}{2} \left(1-\sqrt{1+\frac{87049+36036 \sqrt{6}-6750 \sqrt{\frac{3026695259}{9112500}+\frac{34320041 \sqrt{\frac{2}{3}}}{84375}}}{3375}}\right), $$
 and that the the evaluation of the $z$- and $a$- and $b$-values 
 reduce in the specified manner. 
\end{proof}

 The new Ramanujan-type series highlighted above is 
 inspired by Baruah and Berndt's Ramanujan-type series with summands involving quartic expressions, 
 including the following results introduced in \cite{BaruahBerndt2010}: 
\begin{align*}
	\frac{3\sqrt{2} + \sqrt{5} + 2}{\pi} 
 	 & = \sum_{n=0}^{\infty} 
 \frac{ \left( \frac{1}{2} \right)_{n}^3 }{ \left( n! \right)^3} 
 \left( (3+\sqrt{5}) (2 + \sqrt{5}) (3\sqrt{2} - \sqrt{5} - 2) \right)^{2n}\notag\\
 & \ \ \ \ \ \big( 2 \sqrt{10}
	- 3 \sqrt{5} + 5 \sqrt{2} - 4 + (15 \sqrt{2} + 6 \sqrt{10} - 6 \sqrt{5}) n\big), \\
 \frac{\sqrt{6}+\sqrt{2}+1}{\pi } 
 & = \sum_{n=0}^{\infty} 
 \frac{ \left( \frac{1}{2} \right)_{n}^{3} }{ \left( n! \right)^{3} } 
 \Big( \left( 6 \sqrt{3}+3 \sqrt{6}-6 \right) 
 n + 2 \sqrt{3}+\sqrt{6}-3-\sqrt{2} \Big) \\ 
 & \ \ \ \ \ \left(8 \left(\sqrt{2}+1\right)^2 \left(\sqrt{3}-\sqrt{2}\right)^3 
 \left(2-\sqrt{3}\right)^3\right)^n, \\
 \frac{8}{\pi } 
 & = \sum_{n = 
 0}^{\infty} \frac{ \left( \frac{1}{2} \right)_{n}^{3} }{(n!)^3} 
 \left( \frac{1 - \sqrt{5}}{4} \right)^{3 n} 
 \left(\frac{\sqrt{5}+1}{2} - \sqrt{\frac{\sqrt{5}+1}{2} }\right)^{6 n} \\ 
 & \ \ \ \ \ \Bigg(2 \left(\left(15+5 \sqrt{5}\right) \sqrt{\sqrt{5}+1}-7 \sqrt{10}-5 \sqrt{2}\right) n + \\ 
 & \ \ \ \ \ \left(9+3 \sqrt{5}\right)
 \sqrt{\sqrt{5}+1}-7 \sqrt{2}-5 \sqrt{10}\Bigg). 
\end{align*}

\section{Conclusion}
 We conclude with two open problems concerning values shown in 
 Table \ref{sqrtqqtqabqlqe}. 

\begin{table}[t]
\centering

 \begin{tabular}{ | c | c | }
 \hline 
 $N$ & $r$ \\ \hline 
 11 & 33 \\ \hline 
 13 & 39 \\ \hline 
 19 & 57 \\ \hline 
 31 & 93 \\ \hline 
 33 & 99 \\ \hline 
 59 & 177 \\ 
 \hline
 \end{tabular}

 \ 

\caption{The $N$-indices for the alternating Ramanujan-type series of level 3 conjectured by 
 Aldawoud \cite{Aldawoud2012} and given in Tables 3.7 and 3.8 in \cite{Aldawoud2012}, 
 along with the corresponding $r$-values according to Theorem \ref{maintheorem}. }\label{sqrtqqtqabqlqe} 
\end{table}

 The Ramanujan-type series correpsonding to the $N = 13$ case of Table 3.7 in 
 \cite[p.\ 28]{Aldawoud2012} can be proved via the $r = 39$ case of Theorem 
 \ref{maintheorem}, but this would require closed forms for 
 $\lambda^{\ast}(39)$ and $\alpha(39)$. 
 How can the values of $\lambda^{\ast}(39)$ 
 and $\alpha(39)$ be evaluated in closed form? 
 We would also require closed forms for 
 $\lambda^{\ast}(177)$ and $\alpha(177)$, 
 to apply Theorem \ref{maintheorem}
 to prove Aldawoud's conjecture 
 concerning the $N = 59$ case shown in Table 3.7 
 in \cite[p.\ 28]{Aldawoud2012}. 
 How can the values of $\lambda^{\ast}(177)$ 
 and $\alpha(177)$ be evaluated in closed form? 

\subsection*{Acknowledgements}
 The author is grateful to acknowledge support from a Killam Postdoctoral Fellowship and is thankful to Shane Chern, who has provided the author 
 with extremely useful comments and insights concerning the material in this article. 
 The author also wants to thank Christophe Vignat and Tanay 
 Wakhare for useful and engaging discussions. 

\bibliographystyle{amsplain}

\end{document}